\documentclass{amsart}
\usepackage{latexsym,amssymb,hyperref,url,}

\newtheorem{theorem}{Theorem}[section]
\newtheorem{lemma}[theorem]{Lemma}
\newtheorem{definition}[theorem]{Definition}

\newcommand{\Gal}{\operatorname{\mathrm{Gal}}}
\newcommand{\Cay}{\operatorname{\mathrm{Cay}}}
\newcommand{\Irr}{\operatorname{\mathrm{Irr}}}

\title[Eigenvalues of normal Cayley graphs]{Rationality conditions for the eigenvalues of normal finite Cayley graphs}

\author[C. Godsil]{Chris Godsil}
\address{Chris Godsil, Department of Combinatorics and Optimization, University of Waterloo, Canada}
\email{cgodsil@math.uwaterloo.ca}
\author[P. Spiga]{Pablo Spiga}
\address{Pablo Spiga, Dipartimento di Matematica e Applicazioni,
University of Milano-Bicocca, Italy} \email{pablo.spiga@unimib.it}

\thanks{Address correspondence to P. Spiga,
E-mail: pablo.spiga@unimib.it}
\begin{document}
\begin{abstract}
Given a finite group $G$, we say that a subset $C$ of $G$ is power-closed if, for every $x\in C$ and $y\in \langle x\rangle$ with $\langle x\rangle=\langle y\rangle$, we have $y\in C$.

In this paper we are interested in finite Cayley digraphs $\Cay(G,C)$ over $G$ with connection set $C$, where $C$ is a union of conjugacy classes of $G$. We show that each eigenvalue of $\Cay(G,C)$ is integral if and only if $C$ is power-closed. This result will follow from a discussion of some more general rationality conditions on the eigenvalues of $\Cay(G,C)$. 
\end{abstract}

\maketitle

\section{Introduction}\label{intro}

Let $G$ be a finite group and let $C$ be a subset of $G$. The \emph{Cayley digraph} $\Cay(G,C)$ over $G$ with connection set $S$ is the digraph with vertex set $G$ and with $(g,h)$ being a directed arc if and only if $gh^{-1}\in C$. The \emph{eigenvalues} of a digraph  are the eigenvalues of its adjacency matrix.

In this paper we are concerned with some rationality conditions on the eigenvalues of $\Cay(G,C)$ when $C$ is a union of $G$-conjugacy classes. (Cayley digraphs of this form  are sometimes called \emph{normal}.)
In particular, we are interested in the case that each eigenvalue of $\Cay(G,C)$ is rational. Observe that since the eigenvalues of a digraph are  algebraic integers (being the zeros of the characteristic polynomial of a matrix with integer coefficients), we see that if $\lambda$ is a rational eigenvalue of $\Cay(G,C)$, then $\lambda$ is actually an integer.

We say that $C\subseteq G$ is \emph{power-closed} if, for every $x\in C$ and $y\in \langle x\rangle$ with $\langle y\rangle=\langle x\rangle$, we have $y\in C$. 
\begin{theorem}\label{thrm1}
Let $G$ be a finite group and let $C$ be a union of conjugacy classes of $G$. Then each eigenvalue of $\Cay(G,C)$ is an integer if and only if $C$ is power-closed.
\end{theorem}

As every power-closed subset $C$ is inverse-closed (that is, $C=C^{-1}$), if follows that if each eigenvalue of $\Cay(G,C)$ is an integer, then $\Cay(G,C)$ is an undirected graph. Theorem~\ref{thrm1} gives a rather efficient (and linear-algebra-free) test to check when a Cayley digraph has only integer eigenvalues. 

We note that, aside from its inherent interest, there are other reasons to consider this question.
Let $X$ be a graph on $n$ vertices with adjacency matrix $A$. A \textsl{continuous quantum walk}
of graph is specified by the family of matrices
\[
	U(t) := \exp(itA),\quad (t\in\mathbb{R}).
\]
If $u\in V(X)$ we use $e_u$ to denote the standard basis vector in $\mathbb{R}^n$ indexed by $u$.
We say that $X$ is \textsl{periodic} at $u$ if there is a complex scalar $\gamma$ of norm $1$ 
and a positive time $t$ such that
\[
	U(t)e_a = \gamma e_a.
\]
For surveys on this topic see, e.g., \cite{cgsurv,Kempe2003}.
In \cite{Saxena2007} Saxena, Severini and Shparlinski showed that if $X$ was a circulant, then
$X$ was periodic at a vertex if and only if the eigenvalues of $X$ were integers.
Subsequently it was shown in \cite{Godsil2008} that this conclusion held for any vertex-transitive 
graph, not just for circulants. This work has motivated the search for nice classes of vertex-transitive graphs with integer eigenvalues. 

For abelian groups, our theorem is a well-known and classical result of Bridges and 
Mena~\cite[Theorem~$2.4$]{BridgesMena} (observe that for an abelian group $G$,  every subset of $G$ is 
a union of $G$-conjugacy classes). In particular, Theorem~\ref{thrm1} generalizes the work of Bridges 
and Mena by dropping the hypothesis of $G$ being abelian and by replacing it with a natural condition 
on the connection set.

Theorem~\ref{thrm1} will follow at once from a slightly more general theorem. Before giving its statement we need some preliminary notation, which we will use throughout the whole paper, and some observations. Here we follow closely~\cite{Serre}. 

Let $G$ be a finite group and let $C$ be a union of conjugacy classes of $G$. From~\cite{Babai} or~\cite{Persi}, we get that the eigenvalues of $\Cay(G,C)$ are
\begin{equation}\label{eq1}\frac{1}{\chi(1)}\sum_{x\in C}\chi(x),
\end{equation}
as $\chi$ runs through the set of irreducible complex characters of $G$. (We denote this set by $\Irr_{\mathbb{C}}(G)$.) 

Following Serre~\cite[Section~$9.1$]{Serre}, we denote by $R_{\mathbb{C}}(G)$ the subring of the class functions of $G$ generated by $\Irr_{\mathbb{C}}(G)$, that is, $$R_{\mathbb{C}}(G)=\bigoplus_{\chi\in \Irr_{\mathbb{C}}(G)}\mathbb{Z}\chi.$$
More generally, given a field $K$ with $\mathbb{Q}\leq K\leq \mathbb{C}$, we denote by $R_K(G)$ the subring of $R_{\mathbb{C}}(G)$ generated by the characters of the representations of $G$ over $K$.

We let $m$ be the least common multiple of the order of the elements of $G$, $\mathbb{Q}(m)$  the algebraic field obtained by adjoining the $m$th roots of unity to $\mathbb{Q}$ and $\Gamma_{\mathbb{Q}}$ the Galois group of $\mathbb{Q}(m)$ over $\mathbb{Q}$. By a well-known theorem of Brauer~\cite[Theorem~$24$]{Serre}, we have $R_{\mathbb{C}}(G)=R_{\mathbb{Q}(m)}(G)$, that is, every complex irreducible representation of $G$ is realizable over $\mathbb{Q}(m)$. In particular, every  $\chi\in \Irr_{\mathbb{C}}(G)$ has values in $\mathbb{Q}(m)$ and hence, from~\eqref{eq1}, every normal Cayley digraph $\Cay(G,C)$ has all of its eigenvalues in $\mathbb{Q}(m)$.

Now, let $\varepsilon$ be a primitive $m$th root of unity. From a celebrated theorem of Gauss, the $m$th cyclotomic polynomial is irreducible over $\mathbb{Q}$ and hence $\Gamma_{\mathbb{Q}}\cong (\mathbb{Z}/m\mathbb{Z})^*$ (where  $(\mathbb{Z}/m\mathbb{Z})^*$ denotes the invertible elements of   the ring $\mathbb{Z}/m\mathbb{Z}$).  Here we identify $\Gamma_{\mathbb{Q}}$ with  $ (\mathbb{Z}/m\mathbb{Z})^*$ under this isomorphism. More precisely, for $\sigma\in \Gamma_{\mathbb{Q}}$, there exists a unique $t\in (\mathbb{Z}/m\mathbb{Z})^*$ with $\sigma(\varepsilon)=\varepsilon^t$.

Finally, given a field $K$ with $\mathbb{Q}\leq K\leq \mathbb{Q}(m)$, we denote by $\Gamma_K$ the image of $\Gal(\mathbb{Q}(m)/K)$ in $ (\mathbb{Z}/m\mathbb{Z})^*$, and if $t\in \Gamma_K$, we let $\sigma_t$ denote the corresponding element of $\Gal(\mathbb{Q}(m)/K)$.

For $s\in G$ and for an integer $n$, the element $s^n\in G$ depends only on the residue class of $n$ modulo the order of $s$, and hence only on $n$ modulo $m$. Therefore, $s^t$ is defined for each $t\in \Gamma_K$, and the group $\Gamma_K$ induces an action on the underlying set of $G$.

\begin{definition}\label{def}{\rm We say that $g,h\in G$ are $\Gamma_K$-conjugate, if there exists $t\in \Gamma_K$ such that $g$ and $h^t$ are conjugate in $G$.  Clearly, being $\Gamma_K$-conjugate is an equivalence relation in $G$, and we call $\Gamma_K$-conjugacy classes its equivalence classes.

Observe that when $K=\mathbb{Q}(m)$, we have $\Gamma_K=1$ and hence the $\Gamma_K$-conjugacy classes coincide with the $G$-conjugacy classes. Moreover, when $K=\mathbb{Q}$, we have $\Gamma_K=(\mathbb{Z}/m\mathbb{Z})^*$ and hence two elements $g$ and $h$ of $G$ are $\Gamma_K$-conjugate if there exists $t\in (\mathbb{Z}/m\mathbb{Z})^*$ with $g$ conjugate to $h^t$ in $G$.}
\end{definition}

We are finally ready to state the main result of this paper.
\begin{theorem}\label{thrm2}Let $G$ be a finite group, let $C$ be a union of $G$-conjugacy classes, let $m$ be the least common multiple of the order of the elements of $G$ and let $K$ be a field with $\mathbb{Q}\leq K\leq\mathbb{Q}(m)$.  Then each eigenvalue of $\Cay(G,C)$ lies in $K$ if and only if  $C$ is a union of  $\Gamma_K$-conjugacy classes. 
\end{theorem}

\section{Proofs}\label{proofs}

Theorem~\ref{thrm1} follows from Theorem~\ref{thrm2} (applied with $K=\mathbb{Q}$) and the following lemma.
\begin{lemma}\label{lemma1}
Let $G$ be a finite group and let $C$ be a union of $G$-conjugacy classes. Then $C$ is power-closed if and only if $C$ is a union of $\Gamma_{\mathbb{Q}}$-conjugacy classes.
\end{lemma}
\begin{proof}
We first suppose that $C$ is power-closed and we show that $C$ is a union of $\Gamma_{\mathbb{Q}}$-conjugacy classes. Let $x\in C$ and let $y\in G$ be $\Gamma_{\mathbb{Q}}$-conjugate to $x$. Then, by definition, there exists $t\in (\mathbb{Z}/m\mathbb{Z})^*$ with $y^t$ conjugate  to $x$ in $G$, that is, $y^t=x^g$ for some $g\in G$. Now, $x^g\in C$ and $\langle y\rangle=\langle y^t\rangle=\langle x^g\rangle$, thus $y\in C$ because $C$ is power-closed.

Conversely, we suppose that $C$ is a union of $\Gamma_{\mathbb{Q}}$-conjugacy classes and we show that $C$ is power-closed. Let $x\in C$ and $y\in \langle x\rangle$ with $\langle y\rangle=\langle x\rangle$. Then $y=x^{t'}$, for some integer $t'$ coprime to the order $|x|$ of $x$. From Dirichlet's theorem on primes in arithmetic progression, there exists a prime $t\in \{t'+\ell|x|\mid \ell\in \mathbb{Z}\}$ with $t>m$. We get that the residue class of $t$ in $\mathbb{Z}/m\mathbb{Z}$ is invertible. Now $x^t=x^{t'}=y$ and hence $x$ and $y$ are $\Gamma_{\mathbb{Q}}$-conjugate. Thus $y\in C$.
\end{proof}

\begin{proof}[Proof of Theorem~$\ref{thrm2}$]
Suppose that $C$ is a union $C_1\cup\cdots \cup C_\ell$ of $\Gamma_K$-conjugacy classes. From~\eqref{eq1}, we need to show that $\sum_{x\in C}\chi(x)/\chi(1)\in K$, for every $\chi\in \Irr_{\mathbb{C}}(G)$. For simplicity, we write $e_\chi=\sum_{x\in C}\chi(x)/\chi(1)$. As
$$e_\chi=\frac{1}{\chi(1)}\sum_{x\in C}\chi(x)=\left(\frac{1}{\chi(1)}\sum_{x\in C_1}\chi(x)\right)+\cdots+\left(\frac{1}{\chi(1)}\sum_{x\in C_\ell}\chi(x)\right),$$
it suffices to consider the case that $C=C_1$ is a $\Gamma_K$-conjugacy class. In particular, from the definition of $\Gamma_K$-conjugacy class we get $C=(x^{t_0})^G\cup \cdots \cup(x^{t_\ell})^G$, for some $x\in G$ and some $t_0,\ldots,t_\ell\in \Gamma_K$. (We denote by $x^G$ the conjugacy class of $x$ under $G$.) Observe that the action of the group $\Gamma_K$ on $C$ induces a transitive action of $\Gamma_K$ on $\{(x^{t_0})^G,\ldots, (x^{t_\ell})^G\}$.

Fix $\chi\in \Irr_{\mathbb{C}}(G)$ and   let $\rho$ be a representation of $G$ affording the character $\chi$.
Let $t\in \Gamma_K$ and let $\sigma$ be the corresponding element in $\Gal(\mathbb{Q}(m)/K)$. For $s\in G$, let $\omega_1,\ldots,\omega_{\chi(1)}$ be the eigenvalues of $\rho(s)$. As $|s|$ is a divisor of $m$, we get that $\omega_i$ is an $m$th root of unity and hence the eigenvalues of $\rho(s^t)$ are the $\omega_1^t,\ldots,\omega_{\chi(1)}^t$. Thus we have
\begin{equation}\label{eq2}(\chi(s))^\sigma=\left(\sum_{i=1}^{\chi(1)}\omega_i\right)^\sigma=\sum_{i=1}^{\chi(1)}\omega_i^t=\chi(s^t).
\end{equation}
Now applying $\sigma$ to $e_\chi$, using~\eqref{eq2} and recalling that the set $C$ is invariant under taking $t$th powers, we get
$e_\chi^\sigma=e_\chi$. In particular, $e_\chi^\sigma=e_\chi$ for every $\sigma\in \Gal(\mathbb{Q}(m)/K)$. Since $\mathbb{Q}(m)/K$ is a Galois extension, we have $e_\chi\in K$. 
 
\smallskip

Conversely, suppose that each eigenvalue of $\Cay(G,C)$ lies in $K$. Since $C$ is a union of $G$-conjugacy classes, for showing that $C$ is also a union of $\Gamma_K$-conjugacy classes it suffices to prove that, for each $x\in C$ and for each $t\in \Gamma_K$, we have $x^t\in C$. We argue by induction on $|x|$. Clearly, if $|x|=1$, then there is nothing to prove. Now assume that $|x|>1$. Let $\eta\in\mathbb{C}$ be a primitive $|x|$th root of unity, let $\theta:\langle x\rangle\to \mathbb{C}$ be the irreducible character of $\langle x\rangle$ with $\theta(x)=\eta$, and let $\Theta=\mathrm{Ind}_{\langle x\rangle}^G(\theta)$, that is, $\Theta$ is the character of $G$ obtained by inducing $\theta$ from $\langle x\rangle$ to $G$. From~\cite[page~$55$]{Serre}, we have
\begin{equation}\label{eq3}
\Theta(s)=\frac{1}{|x|}\sum_{\substack{y\in G\\y^{-1}sy\in \langle x\rangle}}\theta(y^{-1}sy).
\end{equation}

Since $\Theta$ is a character of $G$, $\Theta$ is an integral linear combination of the irreducible characters of $G$. Moreover, since every eigenvalue of $\Cay(G,C)$ lies in $K$, from~\eqref{eq1} we obtain $\sum_{z\in C}\Theta(z)\in K$. Write $e_\Theta:=|x|\sum_{z\in C}\Theta(z)$. From~\eqref{eq3}, we get
\begin{eqnarray}\label{eq4}\nonumber
e_\Theta&=&\sum_{z\in C}\sum_{\substack{y\in G\\ y^{-1}zy\in \langle x\rangle }}\theta(y^{-1}zy)=\sum_{z\in C}\sum_{i=0}^{|x|-1}\sum_{\substack{y\in G\\ y^{-1}zy=x^i}}\theta(x^i)=\sum_{z\in C}\sum_{i=0}^{|x|-1}\sum_{\substack{y\in G\\ y^{-1}zy=x^i}}\eta^i\\
&=&\sum_{i=0}^{|x|-1}\sum_{z\in C}\sum_{\substack{y\in G\\ y^{-1}zy=x^i}}\eta^i=a_0\eta^0+a_1\eta^1+\cdots +a_{|x|-1}\eta^{|x|-1},
\end{eqnarray} 
where $a_0,\ldots,a_{|x|-1}$ are non-negative integers. More precisely, 
\begin{equation}\label{eq5}
a_i=|\{(z,y)\mid z\in C,y\in G,y^{-1}zy=x^i\}|.
\end{equation} 
Furthermore, $a_1>0$ because $x\in C$.

Now, let $t\in \Gamma_K$ and let $\sigma$ be its corresponding element in $\Gal(\mathbb{Q}(m)/K)$. Applying $\sigma$ on both sides of~\eqref{eq4}  we get
$$e_\Theta=e_\Theta^\sigma=a_0\eta^0+a_1\eta^t+a_2\eta^{2t}+\cdots+a_{|x|-1}\eta^{(|x|-1)t}$$
and hence
\begin{equation}\label{eq6}
(a_0-a_{0})\eta^0+(a_1-a_{t^{-1}})\eta^1+(a_2-a_{2t^{-1}})\eta^2+\cdots +(a_{|x|-1}-a_{(|x|-1)t^{-1}})\eta^{|x|-1}=0,
\end{equation}
where the indices are computed modulo $|x|$. Now, observe that from our induction hypothesis, for every divisor $i$ of $|x|$ with $1<i<|x|$, the elements $x^i$ and $x^{it}$ are  either both in $C$ or both in $G\setminus C$. In the first case,  from~\eqref{eq5}, we have $a_i=a_{it}$. In the second case, $a_{i}=0$ and $a_{it}=0$ and hence again $a_i=a_{it}$. It follows that the only summands in~\eqref{eq6} that are possibly not zero correspond to the primitive $|x|$th roots of unity. Therefore ~\eqref{eq6} gives rise to the linear equation $$\sum_{\substack{i=0\\\mathrm{Gcd}(i,|x|)=1}}^{|x|-1}(a_i-a_{it^{-1}})\eta^i=0.$$  From a celebrated theorem of Gauss, $\left(\eta^i\mid 0\leq i\leq |x|-1, {\mathrm{Gcd}(i,|x|)=1}\right)$ is a basis for $\mathbb{Q}(\eta)$ over $\mathbb{Q}$ and hence $a_i=a_{it^{-1}}$, for every $i$. In particular, $a_t=a_1>0$ and hence $x^t\in C$ from~\eqref{eq5}.
\end{proof}

\thebibliography{10}
\bibitem{Babai}L.~Babai, Spectra of Cayley Graphs, \textit{J. Combin. Theory B.} \textbf{2}, (1979),
180--189.

\bibitem{BridgesMena}W.~G.~Bridges, R.~A.~Mena, Rational {$G$}-matrices with rational eigenvalues, \textit{J. Combin. Theory Ser. A} \textbf{32}, (1982).
 

\bibitem{Persi}P.~Diaconis, M.~Shahshahani, Generating a Random Permutation with Random Transpositions, \textit{Zeit. f\"{u}r Wahrscheinlichkeitstheorie} \textbf{57} (1981), 159--179.


\bibitem{Godsil2008}
C.~Godsil, Periodic Graphs, \textit{Electronic J. Combinatorics} \textbf{18}(1):$\backslash$\#23, June 2011.

\bibitem{cgsurv}
C.~Godsil, State transfer on graphs, \textit{Discrete Math.} \textbf{312} (2012), 129--147.

\bibitem{Kempe2003}
J.~Kempe, Quantum random walks - an introductory overview,
  \textit{Contemporary Physics} \textbf{44} (2003), 307--327. \href{http://arxiv.org/abs/quant-ph/0303081}{arxiv:0303081}.

\bibitem{Saxena2007}N.~Saxena, S.~Severini, I.~Shparlinski, Parameters of integral  circulant graphs and periodic quantum dynamics, \textit{International Journal on Quantum Computation} \textbf{5}(3) (2007), 417--430. \href{http://arxiv.org/abs/quant-ph/0703236}{arXiv:quant-ph/0703236}.


\bibitem{Serre}J-P.~Serre, Linear Representations of Finite Groups, Graduate Texts in Mathematics \textbf{42},  Springer-Verlag, 1977.
\end{document}